\documentclass[11pt,a4paper]{amsart}
\usepackage{amsfonts,mathtools,xcolor}
\usepackage{amsthm,latexsym,breqn,babel}
\usepackage{amssymb}
\usepackage{amsmath}
\usepackage{enumitem}
\usepackage{amscd}
\usepackage[latin2]{inputenc}
\usepackage{t1enc}
\usepackage[mathscr]{eucal}
\usepackage{indentfirst}
\usepackage{graphicx}
\usepackage{graphics}
\usepackage{pict2e}
\usepackage{epic}
\usepackage{babel}[francais]

\numberwithin{equation}{section}
\usepackage[margin=2.5cm]{geometry}
\usepackage{times}

\usepackage{breqn}

\newcommand{\Z}{\mathbb{Z}}
\newtheorem{thm}{Th\'eor\`eme}

\newtheorem{cor}{Corollaire}

\newtheorem{exemple}{Exemple}
\newtheorem{lem}{Lemme}
\newtheorem{prop}{Proposition}
\newtheorem{prob}{Probl\`eme}

\newtheorem*{ques}{Question}

\newcommand{\resp}{{\it resp. }}
\newcommand{\E}{{\mathcal E }}

\renewcommand{\phi}{\varphi}
\renewcommand{\cong}{\simeq}
\newcommand{\ind}{{\rm ind }}
\newcommand{\cd}{{\rm cd } }

\begin{document}

\title[]{Autour de la d\'ecomposition des alg\`ebres d'exposant 2 sur les extensions multiquadratiques}

\author{Demba Barry et Ahmed Laghribi}
\address{D\'{e}partement de Math\'{e}matique et Informatique , Universit\'{e} des Sciences, des Techniques et des technologies de Bamako, Colline de Badalabougou BP E3206, Bamako, Mali}
\email{barry.demba@gmail.com}
\thanks{Le premier auteur remercie Universit\'e d'Artois pour son financement et son hospitalit\'e (en 2018) durant une partie de ce travail}

%\author[2]{Ahmed Laghribi}
\address{Univ. Artois, UR 2462, Laboratoire de Mathématiques de Lens (LML), F-62300 Lens, France}
\email{ahmed.laghribi@univ-artois.fr}

\date{08 avril, 2022}

\begin{sloppypar}

\begin{abstract}
Pour les alg\`ebres simples centrales d'exposant $2$ sur des corps de caract\'eristique $2$ et de $2$-dimension cohomologique \'egale \`a 2, nous \'etudions la notion de d\'ecomposition adapt\'ee \`a certaines extensions multiquadratiques du centre. Les r\'esultats obtenus \'etendent plusieurs propri\'et\'es remarquables aux extensions multiquadratiques de degr\'e de s\'eparabilit\'e au plus $4$. Nous \'etendons aussi \`a la caract\'eristique $2$ un r\'esultat de Elman-Lam-Tignol-Wadsworth en construisant une alg\`ebre d'exposant $2$ et de degr\'e $8$ contenant une extension triquadratique s\'eparable mais qui n'admet aucune d\'ecomposition adapt\'ee \`a cette extension. Comme application nous donnons une preuve \'el\'ementaire de la non excellence des extensions biquadratiques s\'eparables. 

\bigskip
\noindent \textbf{Mots cl\'es:} Alg\`ebre simple centrale, forme quadratique, extension multiquadratique, excellence, dimension cohomologique.

\medskip
\noindent \textbf{MSC:}  11E08, 11E81, 16K20, 13N05.
\end{abstract}

\maketitle

\section{Introduction}

Soit $F$ un corps commutatif. Une alg\`ebre de quaternions sur $F$ est une alg\`ebre de la forme
\[
[\alpha,\beta) = F[i,j: i^2+i=\alpha, j^2=\beta, jij^{-1}= i+1]
\]
pour certains $\alpha\in F$ et $\beta\in F^\times:=F\setminus \{0\}$ si la caract\'eristique $\operatorname{car}(F)$ de $F$ est $2$, et
\[
(\alpha,\beta) = F[i,j: i^2=\alpha, j^2=\beta, ij= -ji]
\] 
pour certains $\alpha,\beta\in F^\times$ si $\operatorname{car}(F)\ne2$.

Soit $A$ une $F$-alg\`ebre simple centrale d'exposant $2$. L'alg\`ebre est dite \emph{d\'ecomposable}  si $A\simeq A_1\otimes A_2$, o\`u $A_1$ et $A_2$ sont deux alg\`ebres simples centrales sur $F$ et toutes les deux non isomorphes \`a $F$. On dit que $A$ est \emph{totalement d\'ecomposable} si elle est isomorphe \`a un produit tensoriel d'alg\`ebres de quaternions sur $F$. Soient $K_1,\ldots, K_r$ des extensions quadratiques de $F$ contenues dans $A$ telles que $K=K_1\otimes \cdots\otimes K_r$ soit un corps. On dira que $A$ admet une \emph{d\'ecomposition adapt\'ee} \`a $K$ s'il existe des alg\`ebres de quaternions $Q_1,\ldots, Q_r$, avec $K_i\subset Q_i$, telles que $A\simeq Q_1\otimes \cdots\otimes Q_r\otimes A'$ pour une certaine sous-alg\`ebre $A'$. Dans ce papier, nous nous int\'eressons \`a la question de d\'ecomposition adapt\'ee dans le cas o\`u la caract\'eristique de $F$ est $2$ et la $2$-dimension cohomologique $\operatorname{cd}_2(F)$ de $F$ est \'egale \`a $2$.

Notons que la question de d\'ecomposition des alg\`ebres simples centrales d\'epend de la dimension cohomologique du corps de base. Si $\operatorname{cd}_2(F)>2$, des exemples d'alg\`ebres ind\'ecomposables d'exposant $2$ existent (voir par exemple~\cite{ART} si $\operatorname{car}(F)\ne2$, et~\cite{BCL} si $\operatorname{car}(F)=2$). En plus, soient $A$ une $F$-alg\`ebre simple centrale totalement d\'ecomposable de degr\'e au moins $8$ et $K/F$ une extension quadratique s\'eparable contenue dans $A$. Il est connu que $K$ n'est pas en g\'en\'eral dans une sous-alg\`ebre de quaternions de $A$ (voir~\cite[Corollary 4.5]{B16} si $\operatorname{car}(F)\ne2$, et~\cite[Example 3.8]{BCL} si $\operatorname{car}(F)=2$).

En revanche, supposons $\operatorname{cd}_2(F)\le 2$ et soit $A$ une $F$-alg\`ebre simple centrale d'exposant $2$ et de degr\'e $2^n$ ($n\geq 2$). Il est connu que $A$  est totalement d\'ecomposable, voir~\cite[Th\'eor\`eme 3]{Kah} si $\operatorname{car}(F)\ne2$ et~\cite[Theorem 4.1]{BC} si $\operatorname{car}(F)=2$. En outre, soit $K/F$ une extension s\'eparable quadratique ou biquadratique contenue dans $A$. Il est montr\'e dans~\cite[Theorem 3.3]{B}, si $\operatorname{car}(F)\ne 2$, et~\cite[Theorem 4.2]{BC}, si $\operatorname{car}(F)= 2$, que $A$ admet une d\'ecomposition adapt\'ee \`a $K$. Ce dernier r\'esultat n'est plus vrai pour les extensions triquadratiques. En effet, dans~\cite[Remark 5.8]{ELTW} il est construit en caract\'eristique diff\'erente de $2$ une alg\`ebre $A$ de degr\'e $8$ et d'exposant $2$ de centre $F$, avec $\operatorname{cd}_2(F)\le 2$, contenant une extension triquadratique  $M$ de $F$ telle que $A$ n'admette aucune  d\'ecomposition adapt\'ee \`a $M$. Dans cet article, nous \'etendons cet exemple \`a la caract\'eristique $2$.

\medskip
 
Dans tout le reste de cet article, nous supposons que le corps de base $F$ est de caract\'eristique $2$. Notre but est de prouver le r\'esultat suivant:

\begin{thm} Supposons que $\cd_2(F)\leq 2$. Soint $A$ une $F$-alg\`ebre centrale \`a division d'exposant $2$ et de degr\'e $2^n$ ($n\geq 2$). Soit $K$ une extension multiquadratique de $F$. Supposons qu'on ait l'une des conditions suivantes qui s'excluent mutuellement:\\(1) $K$ est de degr\'e de s\'eparabilit\'e $\leq 2$ avec $[K:F]=2^k \leq 2^n$ et $\ind A_K=2^{n-k}$.%\\(2) $K$ est biquadratique s\'eparable et $\ind A_K=2^{n-2}$.
\\(2) $K$ est de degr\'e de s\'eparabilit\'e $4$ telle que $[K:F]=2^n$ et $A_K$ soit d\'eploy\'ee.\\(3) $K$  est une extension triquadratique de degr\'e de s\'eparabilit\'e $4$, $n\geq 4$ et $\ind A_K=2^{n-3}$.\\Alors, $A$ admet une d\'ecomposition adapt\'ee \`a $K$.
\label{t1}
\end{thm}

Cependant on montre que la d\'ecomposition adapt\'ee n'est plus vraie pour les extensions multiquadratiques de degr\'e de s\'eparabilit\'e $8$:

\begin{prop}
Il existe un corps $F$ de caract\'eristique $2$ v\'erifiant $\cd_2(F)= 2$, une $F$-alg\`ebre $A$ simple centrale de degr\'e $8$ et d'exposant $2$ qui n'est pas \`a division, une extension $K/F$ triquadratique s\'eparable tel que $A_K$ soit d\'eploy\'ee mais $A$ n'admette pas de d\'ecomposition adapt\'ee \`a $K$.
\label{pcex}
\end{prop}

La preuve du th\'eor\`eme \ref{t1} est bas\'ee, d'une part sur les articles \cite{AL, ALO} donnant les noyaux de Witt des extensions multiquadratiques de $F$ de degr\'e de s\'eparabilit\'e au plus $4$, et d'autre part sur la propri\'et\'e d'excellence de certaines extensions de $F$ pour les formes de $I^2_qF$. Ce th\'eor\`eme va au-del\`a de la d\'ecomposition adapt\'ee pour les extensions biquadratiques \'etablie dans \cite[Theorem 4.2]{BC}. Noter que dans le cas particulier o\`u $K$ est purement ins\'eparable, quelques r\'esultats (ind\'ependants de la dimension cohomologique de $F$)  de d\'ecomposition adapt\'ee ont \'et\'e obtenus par  Mammone et Moresi \cite[th\'eor\`emes 4 et 5]{MM95}. Quant \`a la proposition \ref{pcex}, elle donne l'analogue en caract\'eristique $2$ d'un r\'esultat de Elman-Lam-Tignol-Wadsworth \cite[Remark 5.8]{ELTW}. Notre m\'ethode est sp\'ecifique \`a la caract\'eristique $2$ et fait appel \`a un r\'esultat sur les formes diff\'erentielles \cite{AB3} et elle est aussi bas\'ee sur une adaptation de certains arguments de Leep et Smith \cite{LS}. Comme application, nous utilisons la proposition \ref{pcex} pour construire un exemple \'el\'ementaire de la non excellence des extensions biquadratiques s\'eparables.

\section{Quelques rappels}
Pour plus de d\'etails sur les notions utilis\'ees dans cet article on renvoie \`a \cite{EKM}. Rappelons qu'une $F$-forme quadratique $\phi$ s'\'ecrit \`a isom\'etrie pr\`es:$$\phi \cong [a_1, b_1]\perp \cdots \perp [a_r, b_r] \perp \left<c_1\right>\perp \cdots \perp \left<c_s\right>,$$o\`u $[a, b]$ (\resp $\left<c\right>$) d\'esigne la forme quadratique $ax^2+xy+by^2$ (\resp la forme quadratique $cx^2$). La forme $\phi$ est dite singuli\`ere (\resp non singuli\`ere) si $s>0$ (\resp $s=0$). 

Pour $a, b, c\in F$ avec $c\neq 0$, on note $[a; c; b]$ la forme quadratique $ax^2+cxy+by^2$. Rappelons que $[a; c; b]$ est isom\'etrique \`a $[a,c^{-2}b]$.

Une forme quadratique $\phi$ d'espace sous-jacent $V$ est dite isotrope s'il existe $v\in V\setminus \{0\}$ tel que $\phi(v)=0$. Sinon, $\phi$ est dite anisotrope. La forme $\phi$ se d\'ecompose de mani\`ere unique \`a isom\'etrie pr\`es comme suit:$$\phi \cong i\times [0, 0]\perp j\times \left<0\right> \perp \phi_{an},$$o\`u $\phi_{an}$ est une forme quadratique anisotrope, qu'on appelle la partie anisotroipe de $\phi$. L'entier $i$ s'appelle l'indice de Witt de $\phi$, on le note $i_W(\phi)$.

Pour tous $a_1, \cdots, a_n\in F^{\times}$, on d\'esigne par $\left<a_1, \cdots, a_n\right>_b$ la forme bilin\'eaire diagonale donn\'ee par:$$((x_1, \cdots, x_n),(y_1, \cdots, y_n))\mapsto \sum_{i=1}^na_ix_iy_i.$$

Dans cet article le terme ``forme bilin\'eaire'' signifie forme bilin\'eaire sym\'etrique r\'eguli\`ere. On note $W_q(F)$ (\resp $W(F)$) le groupe de Witt des $F$-formes quadratiques non singuli\`eres (\resp l'anneau de Witt des $F$-formes bilin\'eaires). Le groupe $W_q(F)$ est muni d'une action de $W(F)$-module induite de fa\c con naturelle par le produit tensoriel: $B\otimes \phi(v\otimes v')=B(v,v)\phi(v')$ pour $v\in V$ et $v'\in V'$, o\`u $B$ est une forme bilin\'eaire d'espace sous-jacent $V$ et $\phi$ est une forme quadratique non singuli\`ere d'espace sous-jacent $V'$ \cite{B1}. Cette action induit une filtration sur $W_q(F)$ par les sous-modules $(I^n_qF)_{n\geq 1}$ comme suit: $I^1_qF=W_q(F)$ et pour $n\geq 2$ on prend $I^n_qF=I^{n-1}F\otimes W_q(F)$, o\`u $I^kF$ est la $k$-i\`eme puissance de l'id\'eal fondamental $IF$ de $W(F)$ form\'e des $F$-formes bilin\'eaires de dimension paire (on prend $I^0F=W(F)$). Le sous-module $I^n_qF$ est additivement engendr\'e par les $n$-formes quadratiques de Pfister $\left<\left<a_1, \cdots, a_{n-1}\right>\right>_b\otimes [1, b]$, o\`u $\left<\left<a_1, \cdots, a_{n-1}\right>\right>_b$ est la $(n-1)$-forme bilin\'eaire de Pfister $\left<1, a_1\right>_b\otimes \cdots \otimes\left<1, a_{n-1}\right>_b$. Soit $\overline{I_q^n}F$ (\resp $\overline{I^n}F$) le quotient $I^n_qF/I^{n+1}_qF$ (\resp le quotient $I^nF/I^{n+1}F$) pour tout entier $n\geq 1$.
\vskip1mm

Soient $\phi$ et $\psi$ deux formes quadratiques d'espaces sous-jacents $V$ et $W$, respectivement. On dit que $\phi$ est domin\'ee par $\psi$, qu'on note $\phi \prec \psi$, s'il existe une injection $F$-lin\'eaire $\sigma: V\longrightarrow W$ v\'erifiant $\psi(\sigma(v))=\phi(v)$ pour tout $v\in V$.

Une forme quadratique $\phi$ est dite une voisine de Pfister s'il existe une forme de Pfister $\pi$ telle que $2\dim \phi>\dim \pi$ et $a\phi \prec \pi$ pour un certain $a\in F^{\times}$. Dans ce cas, $\pi$ est unique et  $\phi$ est isotrope si et seulement si $\pi$ est isotrope.

Pour tout entier $m\geq 1$, soit $\Omega_F^m=\wedge^m\Omega_F^1$ l'espace des $m$-formes diff\'erentielles sur $F$, o\`u $\Omega_F^1$ est le $F$-espace vectoriel engendr\'e par les symboles $dx$ v\'erifiant les relations: $d(x+y)=dx +dy$ et $d(xy)=xdy +ydx$ pour tous $x, y\in F$. On pose $\Omega_F^0=F$ et $\Omega_F^k=0$ si $k<0$. L'application d'Artin-Schreier classique $\wp:F\longrightarrow F$, $x \mapsto x^2-x$, s'\'etend en l'op\'erateur d'Artin-Schreier $\wp: \Omega_F^m\longrightarrow \Omega_F^m/d\Omega_F^{m-1}$ donn\'e par: $$\sum_{finie}c_i \frac{da_{i_1}}{a_{i_1}}\wedge \cdots \wedge \frac{da_{i_m}}{a_{i_m}} \mapsto \sum_{finie}\overline{\wp(c_i)\frac{da_{i_1}}{a_{i_1}}\wedge \cdots \wedge \frac{da_{i_m}}{a_{i_m}}},$$o\`u $d:\Omega_F^{m-1}\longrightarrow \Omega_F^m$ est l'op\'erateur diff\'erentiel donn\'e par: $$d(xda_1\wedge \cdots \wedge da_{m-1})=dx\wedge da_1\wedge \cdots \wedge da_{m-1}.$$ On note $\nu_F(m)$ et $H_2^{m+1}(F)$ le noyau et le conoyau de $\wp$, respectivement. Un c\'el\`ebre r\'esultat de Kato \cite{K} donne le lien entre les formes quadratiques (formes bilin\'eaires) et les formes diff\'erentiells comme suit: 
							
\begin{eqnarray}\nonumber
e_n: \overline{I}^{n+1}_qF & \rightarrow & H_2^{n+1}(F) \\\label{Katoiso}
\overline{\left<\left< a_1, \hdots, a_n\right>\right>_b\otimes [1, b]} & \mapsto & \overline{b\frac{d a_1}{a_1} \wedge \hdots \wedge \frac{d a_n}{a_n}}.
\end{eqnarray}

\begin{eqnarray}\nonumber
f_n: \overline{I}^nF & \rightarrow & \nu_F(n) \\\label{Katoiso2}
\overline{\left<\left< a_1, \hdots, a_n\right>\right>_b} & \mapsto & \frac{d a_1}{a_1} \wedge \hdots \wedge \frac{d a_n}{a_n}.
\end{eqnarray}

En particulier, de l'isomorphisme (\ref{Katoiso2}), on d\'eduit que $\nu_F(n)$ est additivement engendr\'e par les symboles logarithmiques $\frac{d a_1}{a_1} \wedge \hdots \wedge \frac{d a_n}{a_n}$.

On sait par \cite[(5.12)]{AB} que $H^{n+1}_2(F)\cong H^1({\rm Gal}(F_s/F), \nu_{F_s}(n))$, o\`u $F_s$ est la cl\^oture s\'eparable de $F$. Dans \cite[Appendix 101]{EKM} ce groupe est not\'e $H^{n+1, n}(F, \Z/2\Z)$. La $2$-dimension cohomologique $\cd_2(F)$ de $F$ est le plus petit entier tel que pour toute extension finie $L/F$ et tout entier $n\geq \cd_2(F)$, on ait $H^{n+1, n}(L, \Z/2\Z)=0$, ce qui \'equivaut \`a $\overline{I^{n+1}_q}L=0$ par l'isomorphisme (\ref{Katoiso}). Cela donne par le Hauptsatz d'Arason-Pfister $I^{n+1}_qL=0$.

Pour toute forme $\phi$ non singuli\`ere de dimension $2m$, on note $\Delta(\phi)$ et $C(\phi)$ son invariant d'Arf et son alg\`ebre de Clifford. Plus explicitement, si $\phi \cong a_1[1, b_1] \perp \cdots \perp a_m[1, b_m]$, alors $\Delta(\phi)= \sum_{i=1}^m b_i+\wp(F)\in F/\wp(F)$ et $C(\phi)\cong \otimes_{i=1}^m[b_i, a_i)$. En particulier, $C(\phi)$ est une $F$-alg\`ebre simple centrale d'exposant $2$ et de degr\'e $2^m$. Si de plus $\phi \in I^2_qF$, alors $C(\phi)\cong M_2(D_{\phi})$, o\`u $D_{\phi}$ est un produit tensoriel de $m-1$ alg\`ebres de quaternions.

On note ${\rm Br}(F)$ le groupe de Brauer de $F$. Pour une $F$-alg\`ebre simple centrale $A$, l'entier $\sqrt{\dim_F A}$ s'appelle le degr\'e de $A$ et on le note $\deg A$. L'indice de $A$ est le degr\'e de l'alg\`ebre \`a division Brauer-\'equivalente \`a $A$, on le note $\ind A$.

L'hypoth\`ese $\cd_2(F)\leq 2$ qui nous int\'eresse dans cet article permet d'avoir des informations tr\`es pr\'ecises concernant les $F$-alg\`ebres simples centrales d'exposant $2$:

\begin{thm} (\cite[Theorem 4.1]{BC}) Supposons que $\cd_2(F)\leq 2$. Pour toute $F$-forme quadratique $\phi \in I^2_qF$, l'alg\`ebre $D_{\phi}$ est \`a division si et seulement si $\phi$ est anisotrope. De plus, toute $F$-alg\`ebre simple centrale d'exposant $2$ et de degr\'e $2^m$ est isomorphe \`a une $F$-alg\`ebre $D_{\phi}$ pour une certaine $F$-forme quadratique $\phi \in I^2_qF$ de dimension $2m+2$.
\label{propdec}
\end{thm}

On obtient le corollaire suivant:

\begin{cor} Supposons que $\cd_2(F)\leq 2$ et soit $\phi \in I^2_qF$. Si $\ind D_{\phi}=2^l$, alors $\dim \phi_{an}=2l+2$.
\label{cordec}
\end{cor}

\noindent{\bf Preuve.} Posons $\dim \phi=2m+2$. On sait que $C(\phi) \cong M_2(D_{\phi})$ avec $D_{\phi}$ un produit de $m$ alg\`ebres de quaternions, donc $\deg D_{\phi}=2^m$. Posons $\ind D_{\phi}=2^l$. Alors, $D_{\phi}\cong M_{2^{m-l}}(A)$ pour une $F$-alg\`ebre centrale \`a division $A$ de degr\'e $2^l$. Par le th\'eor\`eme pr\'ec\'edent, il existe $\phi'$ une $F$-forme quadratique anisotrope de dimension $2l+2$ telle que $C(\phi')\cong M_2(A)$. Ainsi, on obtient 
\begin{eqnarray*}
C(\phi) &\cong & M_2(D_{\phi})\\ & \cong & M_{2^{m-l}}(M_2(A))\\ &\cong & M_{2^{m-l}}(C(\phi'))\\ & \cong & C((m-l)\times [0, 0] \perp \phi').
\end{eqnarray*}
Par un r\'esultat de Sah \cite{S}, on a $\phi \perp -\phi' \in I^3_qF$. Comme $I^3_qF=0$ (car $\cd_2(F)\leq 2$), on d\'eduit que $\phi$ est Witt-\'equivalente \`a $\phi'$, ce qui implique $\phi_{an}\cong \phi'$ est de dimension $2l+2$.\qed 
\medskip

On finit cette section par rappeler la d\'efinition des formes r\'esiduelles d'une forme quadratique sur un corps ${\mathcal F}$ Hens\'elien pour une valuation discr\`ete. Soient $A$ l'anneau de la valuation de ${\mathcal F}$, $\pi$ une uniformisante et $\kappa =A/\pi A$ le corps r\'esiduel. Soit $\theta$ une forme quadratique sur ${\mathcal F}$ anisotrope (\'eventuellement singuli\`ere) d'espace sous-jacent $V$. Pour tout $i\in \Z$, on associe l'ensemble $$V_i=\{v\in V\mid \theta(v)\in \pi^iA\},$$ c'est un $A$-module d'apr\`es \cite[page 342]{MMW91}. On obtient les $\kappa$-formes quadratiques $\overline{\theta_0}$ et $\overline{\theta_1}$ d\'efinies par:
$$\begin{tabular}{cccl}
$\overline{\theta_i}:$ & $V_i/V_{i+1}$ & $\longrightarrow$ & $\kappa $\\ & $v+V_{i+1}$ & $\mapsto$ & $\overline{\pi^{-i}\theta(v)}$.
\end{tabular}$$

La forme $\overline{\theta_0}$ (\resp $\overline{\theta_1}$) est appel\'ee la premi\`ere forme r\'esiduelle de $\theta$ (\resp la seconde  forme r\'esiduelle de $\theta$). \`A noter que lorsque $\theta$ est non singuli\`ere, on a $\dim \theta= \dim \overline{\theta_0} +\dim \overline{\theta_1}$ \cite[Theorem 1]{MMW91}. 

Soient $u, v \in {\mathcal F}$ des unit\'es et $\epsilon \in \Z$ tels que la forme $\lambda:=[u, \pi^{\epsilon}v]$ soit anisotrope. Par l'in\'egalit\'e de Schwarz \cite[(4)]{MMW91}, \cite[Lemma 2.2]{T}, on a n\'ecessairement $\epsilon\leq 0$. Voici quelques cas de calcul des formes r\'esiduelles de $\lambda$:\vskip1mm
(A) Si $\epsilon=0$, alors la premi\`ere forme r\'esiduelle est $[\overline{u}, \overline{v}]$ et la seconde forme r\'esiduelle est $0$.\vskip1mm
(B) Si $\epsilon$ est impair, alors la premi\`ere forme r\'esiduelle est $\left<\overline{u}\right>$ et la seconde forme r\'esiduelle est $\left<\overline{v}\right>$.\vskip1mm
(C) Si $\epsilon$ est pair non nul, alors $\lambda\cong [u; \pi^{\frac{-\epsilon}{2}}; v]$. Si la forme $\left<\overline{u}, \overline{v}\right>$ est anisotrope, alors la premi\`ere forme r\'esiduelle de $\lambda$ est $\left<\overline{u}, \overline{v}\right>$, et la seconde forme r\'esiduelle est $0$.

On renvoie \`a \cite[Section 3]{MMW91} pour un algorithme de calcul des formes r\'esiduelles de toute ${\mathcal F}$-forme quadratique anisotrope.

\section{R\'esultats pr\'eliminaires}

Rappelons qu'une extension $K/F$ est dite {\it excellente} si pour toute $F$-forme quadratique anisotrope $\phi$, il existe une $F$-forme quadratique $\psi$ tel que $(\phi_K)_{an}\cong \psi_K$. 

On dit que $K/F$ est excellente pour les formes de $I^m_qF$ si pour toute forme anisotrope $\phi \in I^m_qF$, il existe une forme $\psi \in I^m_qF$ telle que $(\phi_K)_{an}\cong \psi_K$. On note $\E_m(F)$ l'ensemble des extensions de $F$ qui sont excellentes pour les formes de $I^m_qF$.

Dans cet article on s'int\'eresse \`a l'ensemble $\E_2(F)$ qui est li\'e au probl\`eme de d\'ecomposabilit\'e des alg\`ebres simples centrales d'exposant $2$. \`A ce propos, on donne l'exemple suivant.

\begin{exemple}
(1) Toute extension $K/F$ excellente appartient \`a $\E_2(F)$.\\(2) Toute extension $K/F$ quadratique s\'eparable, ou multiquadratique purement 
ins\'eparable appartient \`a $\E_2(F)$.
\end{exemple}

\noindent{\bf Preuve.} (1) Supposons que $K/F$ soit excellente. Soit $\phi \in I^2_qF$ anisotrope. Alors, il existe $\psi \in W_q(F)$ anisotrope telle que $(\phi_K)_{an}\cong \psi_K$. Par \cite[Lemma 2.1]{BC} on peut supposer $\psi \in I^2_qF$.

(2) Il est bien connu que toute extension quadratique de $F$ est excellente \cite[Lemma 5.4]{HL}. D'apr\`es \cite{H}, toute extension multiquadratique purement ins\'eparable de $F$ est excellente.\qed\vskip1.5mm

On va prouver le r\'esultat suivant donnant d'autres exemples d'\'el\'ements de $\E_2(F)$ et qui servira pour la preuve du th\'eor\`eme \ref{t1}.

\begin{prop} Supposons que $I^3_qF=0$. Soit $K$ une extension de $F$ qui est de l'un des trois types suivants:
\begin{enumerate}
\item[(1)] biquadratique s\'eparable.
\item[(2)] multiquadratique de degr\'e de s\'eparabilit\'e $\leq 2$.
\item[(3)] triquadratique de degr\'e de s\'eparabilit\'e $4$.
\end{enumerate}
Alors, l'extension $K/F$ appartient \`a $\E_2(F)$.
\label{p1}
\end{prop}

On sait qu'il existe un corps $F_0$ de caract\'eristique $2$ et une extension biquadratique s\'eparable de $F_0$ qui n'est pas excellente (cet exemple a \'et\'e construit en premier en caract\'eristique $\neq 2$ puis a \'et\'e \'etendu \`a la caract\'eristique $2$ dans \cite{LM}). On renvoie aussi au corollaire \ref{corexce} o\`u on donne un autre exemple d'extension biquadratique s\'eparable non-excellente. Ainsi, l'assertion (1) de la proposition pr\'ec\'edente montre qu'en g\'en\'eral l'ensemble $\E_2(F)$ est diff\'erent de celui de toutes les extensions excellentes de $F$.

Avant de prouver la proposition \ref{p1}, on aborde un autre probl\`eme qui est consid\'er\'e comme la r\'eciproque (\`a \'equivalence de Witt pr\`es) de l'excellence pour les formes de $I^m_qF$.

\begin{prob} {\bf (Descente en degr\'e $m$).}\\
Soient $K/F$ une extension et $\phi \in I^m_qK \cap {\rm Im}(W_q(F)\longrightarrow W_q(K))$. A-t-on $\phi \in {\rm Im}(I^m_qF\longrightarrow I^m_qK)$?
\label{prob1}
\end{prob}

En g\'en\'eral, la r\'eponse \`a ce probl\`eme est li\'ee au calcul du noyau ${\rm Ker}(H_2^{m+1}(F)\longrightarrow H_2^{m+1}(K))$ dont la complexit\'e d\'epend de la structure de l'extension $K/F$. Voici un calcul qui nous int\'eresse dans cet article.

\begin{thm} (\cite{AL}, \cite{ALO}) Soient $K=F(\sqrt{a_1}, \cdots, \sqrt{a_n})$ une extension purement ins\'eparable de $F$ de degr\'e $2^n$, et $L=F(\alpha, \beta)$ une extension biquadratique s\'eparable avec $\alpha^2+\alpha=a \in F\setminus \wp(F)$ et $\beta^2+\beta=b \in F\setminus \wp(F)$. Soit $M$ l'un des corps suivants: $K$, $K(\alpha)$ et $K(\alpha, \beta)$. Alors, pour tout entier $m\geq 1$ on a:
\[ {\rm Ker}(H_2^{m+1}(F)\longrightarrow H_2^{m+1}(M))=\begin{cases}\sum_{i=1}^n\overline{\frac{da_i}{a_i}\wedge \Omega_F^{m-1}} & \text{si } M=K,\\
\sum_{i=1}^n\overline{\frac{da_i}{a_i}\wedge \Omega_F^{m-1}} + \overline{a\nu_F(m)} & \text{si } M=K(\alpha),\\\sum_{i=1}^n\overline{\frac{da_i}{a_i}\wedge \Omega_F^{m-1}} + \overline{a\nu_F(m)} + \overline{b\nu_F(m)} & \text{si } M=K(\alpha, \beta).
\end{cases}\]
\label{p2}
\end{thm}

En utilisant les isomorphismes de Kato (\ref{Katoiso}) et (\ref{Katoiso2}), le th\'eor\`eme pr\'ec\'edent se traduit dans le langage des formes quadratiques comme suit:

\begin{cor}
On garde les m\^emes notations et hypoth\`eses que dans le th\'eor\`eme \ref{p2}. Alors:
\[{\rm Ker}(\overline{I_q^{m+1}}F\longrightarrow \overline{I_q^{m+1}}M)=\begin{cases} \sum_{i=1}^n\overline{\left<\left<a_i\right>\right>_b\otimes I^m_qF} & \text{si } M=K,\\
\sum_{i=1}^n\overline{\left<\left<a_i\right>\right>_b\otimes I^m_qF} + \overline{I^mF\otimes [1,a]} & \text{si } M=K(\alpha),\\ \sum_{i=1}^n\overline{\left<\left<a_i\right>\right>_b\otimes I^m_qF} + \overline{I^mF\otimes [1,a]} + \overline{I^mF\otimes [1,b]} & \text{si } M=K(\alpha, \beta).
\end{cases}\]
\label{c1}
\end{cor}

En proc\'edant par les m\^emes arguments utilis\'es dans \cite[Section 6]{ALO}, apr\`es s'\^etre ramen\'e \`a un corps de base ayant une $2$-base finie, on obtient du corollaire \ref{c1}:

\begin{cor} On garde les m\^emes notations et hypoth\`eses que dans le th\'eor\`eme \ref{p2}. Alors:
\[{\rm Ker}(I_q^{m+1}F\longrightarrow I_q^{m+1}M)=\begin{cases} \sum_{i=1}^n\left<\left<a_i\right>\right>_b\otimes I^m_qF & \text{si } M=K,\\
\sum_{i=1}^n\left<\left<a_i\right>\right>_b\otimes I^m_qF + I^mF\otimes [1,a] & \text{si } M=K(\alpha),\\ \sum_{i=1}^n\left<\left<a_i\right>\right>_b\otimes I^m_qF + I^mF\otimes [1,a] + I^mF\otimes [1,b] & \text{si } M=K(\alpha, \beta).
\end{cases}\]
\label{c2}
\end{cor}

Pour ce qui est du probl\`eme \ref{prob1}, on obtient la r\'eponse suivante:

\begin{cor} On garde les m\^emes notations et hypoth\`eses que dans le th\'eor\`eme \ref{p2}. Alors, le probl\`eme \ref{prob1} a une r\'eponse positive pour l'extension $M/F$, c'est-\`a-dire, on a$$I^m_qM\cap {\rm Im}(W_q(F)\longrightarrow W_q(M))={\rm Im}(I^m_qF \longrightarrow I^m_qM).$$  
\label{c3}
\end{cor} 

\noindent{\bf Preuve.} L'inclusion $\supset$ est \'evidente. Montrons alors l'inclusion $\subset$. La preuve est la m\^eme pour les corps $K, K(\alpha)$ et $K(\alpha, \beta)$. On se limite \`a consid\'erer le cas $M=K$. Soient $\phi \in I^m_qK$ et $\psi \in W_q(F)$ tel que $\phi \sim \psi_{K}$. Soit $p$ tel que $\psi \in I^p_qF$. Si $p\geq m$, alors il n'y a rien \`a prouver. Supposons que $p<m$. La condition $\phi \sim \psi_{K}$ implique $\psi +I^{p+1}_qF \in {\rm Ker}(\overline{I_q^{p}}F\longrightarrow \overline{I_q^{p}}K)$. Par le corollaire \ref{c1}, on a $\psi + \sum_{i=1}^n\left<\left<a_i\right>\right>_b\otimes \phi_i \in I^{p+1}_qF$ pour certaines formes $\phi_i \in I^{p-1}_qF$. Soit $\psi'= \psi + \sum_{i=1}^n\left<\left<a_i\right>\right>_b\otimes \phi_i \in I^{p+1}_qF$. On a $\psi'_{K}\sim \psi_{K} \sim \phi$. Si $m=p+1$ alors on a l'affirmation souhait\'ee. Si $m>p+1$, alors $\psi' +I^{p+2}_qF\in  {\rm Ker}(\overline{I_q^{p+1}}F\longrightarrow \overline{I_q^{p+1}}K)$. Ainsi de suite on continue le proc\'ed\'e jusqu'\`a avoir une forme $\gamma:=\psi + \sum_{i=1}^n\left<\left<a_i\right>\right>_b\otimes \gamma_i \in I^{m}_qF$ pour certaines formes $\gamma_1, \cdots, \gamma_n \in W_q(F)$. Comme $\gamma \in I^m_qF$ et $\gamma_K \sim \phi$, alors $\phi \in {\rm Im}(I^m_qF \longrightarrow I^m_qK)$.\qed\vskip1.5mm

On a le corollaire suivant:

\begin{cor} On garde les m\^emes notations et hypoth\`eses que dans le th\'eor\`eme \ref{p2}. Si $I^m_qF=0$, alors $I^m_qM\cap {\rm Im}(W_q(F)\longrightarrow W_q(M))=0$.
\label{c6}
\end{cor}

Pour la preuve de la proposition \ref{p1} on aura besoin du lemma suivant. Dans le cas d'une extension quadratique ins\'eparable, ce lemme est une cons\'equence du r\'esultat \cite[Theorem 3.1]{AB2} de Aravire et Baeza sur le $\nu$-invariant en caract\'eristique $2$. Pour garder cet article ``self-contained'', on  donnera une preuve mettant le lien avec le probl\`eme \ref{prob1}.

\begin{lem} Soit $L/F$ une extension quadratique. Si $I^m_qF=0$, alors $I^m_qL=0$.
\label{l1}
\end{lem}

\noindent{\bf Preuve.} Posons $L=F(\alpha)$ tel que $\alpha^2=d\in F\setminus F^2$ ou $\alpha^2+\alpha \in F\setminus \wp(F)$. Soit $s:L\longrightarrow F$ l'application $F$-lin\'eaire donn\'ee par: $s(1)=0$ et $s(\alpha)=1$. Notons $s_*:W_q(L)\longrightarrow W_q(F)$ le transfert induit par $s$. Par \cite[Corollary 34.17]{EKM} on a $s_*(I^m_qL)\subset I^m_qF$. Ainsi, $s_*(I^m_qL)=0$. Soit $\phi \in I^m_qL$. On a $s_*(\phi)\sim 0$.

(1) Supposons que $K/F$ soit s\'eparable. Alors par la suite exacte \cite[sequence (6.1)]{AB} la condition $s_*(\phi)\sim 0$ implique $\phi \sim 0$.

(2) Supposons que $K/F$ soit ins\'eparable. Supposons que $\phi$ soit anisotrope de dimension $\geq 2$. Puisque $s_*(\phi)$ est hyperbolique, elle est en particulier isotrope. Il existe $v\in V\setminus \{0\}$ tel que $s_*(\phi)(v)$, o\`u $V$ est l'espace sous-jacent \`a $\phi$. Par cons\'equent, $b:=\phi(v)\in F$. On a $b\neq 0$ car $\phi$ est anisotrope et $v\neq 0$. Ainsi, $\phi \cong b[1, c] \perp \phi'$ pour $c\in L$ et $\phi'$ une forme quadratique sur $L$. Puisque $[1, c]\cong [1, c^2]$, on peut supposer $c\in F$. Ainsi, $s_*(\phi')=0$. Par r\'ecurrence sur la dimension de $\phi$, on d\'eduit que $\phi'$ est d\'efinie sur $F$. Ainsi, $\phi \in {\rm Im}(W_q(F)\longrightarrow W_q(L))$. Comme $I^m_qF=0$ et $\phi \in I^m_qL$, on d\'eduit par le corollaire \ref{c6} que $\phi=0$, une contradiction \`a notre hypoth\`ese. D'o\`u , $\phi \sim 0$.\qed
\medskip

Le corollaire suivant est imm\'ediat.

\begin{cor} Soit $L/F$ une extension multiquadratique. Si $I^m_qF=0$, alors $I^m_qL=0$.
\label{c5}
\end{cor}

\section*{Preuve de la proposition \ref{p1}} Soit $L$ une extension de $F$ qui est de l'un des trois types suivants:
\begin{enumerate}
\item[(a)] multiquadratique purement ins\'eparable.
\item[(b)] quadratique s\'eparable.
\item[(c)] mixed biquadratique (c'est-\`a-dire, la compos\'ee de deux extensions quadratiques de $F$, l'une s\'eparable et l'autre ins\'eparable).
\end{enumerate}

On rappelle que l'extension $L/F$ est excellente: Le cas (a) est d\^u \`a Hoffmann \cite{H}; (b) est d\^u \`a Hoffmann et Laghribi \cite[Lemma 5.4]{HL}; et (c) a \'et\'e r\'ecemment prouv\'e par Mukhija \cite{LM}.

Soit $F(\alpha)/F$ une extension quadratique s\'eparable avec $\alpha^2+ \alpha= a \in F\setminus \wp(F)$, et soit $K=L(\alpha)$. On suppose $[K:L]=2$ dans les cas (b) et (c). 

La preuve qu'on va produire reprend les arguments utilis\'es dans \cite[Lemma 2.1]{B} montrant que les extensions biquadratiques en caract\'eristique diff\'erente de $2$ satisfont \`a l'excellence pour les formes de $I^2$.

Soit $\phi \in I^2_qF$ anisotrope tel que $\phi_K$ soit isotrope. Il n'y a rien \`a prouver si $\phi_K$ est hyperbolique. Supposons que $\phi_K$ ne soit pas hyperbolique. On commence par montrer que $(\phi_K)_{an}$ est d\'efinie sur $F$.

Puisque $L/F$ est excellente, la forme $(\phi_L)_{an}$ est d\'efinie sur $F$. Si $(\phi_L)_{an}$ est anisotrope sur $K$, alors $(\phi_K)_{an}$ est d\'efinie sur $F$. Supposons que $(\phi_L)_{an}$ soit isotrope sur $K$ d'indice de Witt $r$. Alors, $(\phi_L)_{an} \cong \left< x_1, \cdots, x_r\right>_b\otimes [1, a] \perp \phi'$ avec $x_i \in L^{\times}$ et $\phi'$ une forme quadratique sur $L$ qui est anisotrope sur $K$ (on applique $r$ fois le r\'esultat \cite[Theorem 4.2, page 121]{B}). Puisque $I^3_qL=0$, alors $\left<1, x_1\right>\otimes (\phi_L)_{an} \sim 0$. Ainsi, $(\phi_L)_{an}\cong x_1 (\phi_L)_{an}$, et donc on peut supposer $x_1=1$. Toujours de l'hypoth\`ese $I^3_qL=0$, toute $3$-forme de Pfister sur $L$ est hyperbolique, et par cons\'equent toute forme voisine de Pfister de dimension $5$ sur $L$ est isotrope. Cela implique que $r=1$, c'est-\`a-dire, $(\phi_L)_{an} \cong [1, b] \perp \phi'$. Puisque $\phi'\in {\rm Im}(W_q(F) \longrightarrow W_q(L))$ et que $L/F$ est excellente, il existe $\phi_0\in W_q(F)$ tel que $\phi'\cong (\phi_0)_L$. Comme $\phi'_K$ est anisotrope, alors $(\phi_K)_{an} \cong (\phi_0)_K$. Finalement, en utilisant \cite[Lemma 2.1]{BC}, on peut supposer que $\phi_0\in I^2_qF$.

\section*{Preuve du th\'eor\`eme \ref{t1}}
On suppose $\cd_2(F)\leq 2$. Soit $A$ une $F$-alg\`ebre centrale \`a division d'exposant $2$ et de degr\'e $2^n$. Soit $K$ une extension multiquadratique de $F$ de l'un des trois types d\'ecrits dans le th\'eor\`eme. Posons $[K:F]=2^k$ et \'ecrivons $K=F(\sqrt{a_1}, \cdots, \sqrt{a_r}, \alpha_1, \cdots, \alpha_s)$ tel que $a_i\in F\setminus F^2$ pour tout $1\leq i\leq r$ et $\alpha_j^2+\alpha_j=b_j\in F\setminus \wp(F)$ pour tout $1\leq j\leq s$, avec $r+s=k$.

Puisque $\cd_2(F)\leq 2$, il existe une forme $\phi \in I^2_qF$ anisotrope de dimension $2n+2$ telle que $A=D_{\phi}$ (proposition \ref{propdec}). De plus, on a $I^3_qK=0$. 

(1) Supposons que $K$ soit de degr\'e de s\'eparabilit\'e $\leq 2$ ou triquadratique de degr\'e de s\'eparabilit\'e $4$, ce qui revient \`a dire $s=0$ ou ($r\geq 1$ et $s=1$) ou  ($r=1$ et $s=2$). Supposons $\ind (D_{\phi})_K=2^{n-k}$.

Par le corollaire \ref{cordec} appliqu\'e au corps $K$ (car $\cd_2(K)\leq 2$), on a $\dim (\phi_K)_{an}=2(n-k)+2$. Par la proposition \ref{p1}, il existe une forme $\phi'\in I^2_qF$ de dimension $2(n-k)+2$ telle que $(\phi_K)_{an}\cong \phi'_K$. On a $C(\phi')\cong M_2(B)$ avec $B$ un produit de $n-k$ alg\`ebres de quaternions. Puisque $\phi \perp \phi'\in {\rm Ker}(I^2_qF\longrightarrow I^2_qK)$, on applique le corollaire \ref{c2} pour avoir 
\[
\phi \perp \phi'\sim \begin{cases} \sum_{i=1}^r \left<\left< a_i\right>\right>_b\otimes \phi_i & \text{si } s=0,\\\sum_{i=1}^r \left<\left<a_i\right>\right>_b\otimes \phi_i + \sum_{j=1}^s\rho_j \otimes [1,b_j]& \text{si } (r\geq 1 \;\text{et}\; s=1)\; \text{ou}\;  (r=1 \;\text{et}\; s=2),\end{cases}
\]
pour des formes convenables $\phi_1, \cdots, \phi_r\in W_q(F)$ et $\rho_1,\cdots, \rho_s \in IF$. En passant \`a l'alg\`ebre de Clifford, on obtient dans $\operatorname{Br}(F)$: 
\[
C(\phi) + C(\phi')\sim \begin{cases} \otimes_{i=1}^r [c_i, a_i) & \text{si } s=0,\\(\otimes_{i=1}^r [c_i, a_i))\otimes(\otimes_{j=1}^s[b_j,d_j))& \text{ si } (r\geq 1 \;\text{et}\; s=1)\; \text{ou}\;  (r=1 \;\text{et}\; s=2),\end{cases}
\] pour des scalaires convenables $c_1, \cdots, c_r \in F$ et $d_1, \cdots, d_s\in F^{\times}$. Puisque $B$ est un produit de $n-k$ alg\`ebres de quaternions, on obtient par comparaison des dimensions que $A\cong B\otimes C$, o\`u l'alg\`ebre $C$ est produit de $k$ alg\`ebres de quaternions donn\'ee par:
\[
C=\begin{cases} \otimes_{i=1}^r [c_i, a_i) & \text{si } s=0,\\ (\otimes_{i=1}^r [c_i, a_i))\otimes(\otimes_{j=1}^s[b_j, d_j)) & \text{ si }  (r\geq 1 \;\text{et}\; s=1)\; \text{ou}\;  (r=1 \;\text{et}\; s=2).
\end{cases}
\]
D'o\`u, $A$ admet une d\'ecomposition adapt\'ee \`a $K$.
\medskip

(2) Supposons que $K$ soit de degr\'e de s\'eparabilit\'e $4$, $[K:F]=2^n$ et $A_K\sim 0$. On obtient que $\phi_K$ est hyperbolique (corollaire \ref{cordec}). Par le corollaire \ref{c2}, $\phi \sim \sum_{i=1}^{r}\left<\left<a_i\right>\right>_b\otimes\phi_i + \rho_1\otimes [1, b_1] + \rho_2\otimes [1, b_2]$ pour des formes convenables $\phi_1, \cdots, \phi_{r} \in W_q(F)$ et $\rho_1, \rho_2 \in IF$. On conclut comme dans (1) en passant par l'alg\`ebre de Clifford.

\section{Un contre-exemple \`a la d\'ecomposition adapt\'ee}
Notre but dans cette section est de donner une preuve de la proposition \ref{pcex}. Pour cela on utilisera l'article \cite{LS} de Leep et Smith consacr\'e \`a l'\'etude du noyau de Witt des extensions triquadratiques en caract\'eristique diff\'erente de $2$. Plus exactement, pour une extension triquadratique $F_0(\sqrt{a}, \sqrt{b}, \sqrt{c})$ d'un corps $F_0$ de caract\'eristique diff\'erente de $2$, le r\'esultat \cite[Lemma 4.1]{LS} donne une m\'ethode  pour construire des $F_0(\sqrt{a})$-formes quadratiques de dimension $4$ qui ont la propri\'et\'e d'\^etre d\'efinies sur $F$ et hyperboliques sur $F_0(\sqrt{a}, \sqrt{b}, \sqrt{c})$. Ce r\'esultat clarifie l'exemple \cite[Exemple 5.8(v)]{ELTW} qui a motiv\'e la proposition \ref{pcex} (voir le commentaire \cite[Page 254]{LS}). On commence par \'etablir l'analogue en caract\'eristique $2$ de \cite[Lemma 4.1]{LS}: 

\begin{prop} Soit $M=F(\wp^{-1}(a), \wp^{-1}(b), \wp^{-1}(c))$ une extension triquadratique s\'eparable de $F$, et $K=F(\wp^{-1}(a))$. Posons $\alpha=\wp^{-1}(a)$ et soit $\phi=\beta[1, b] \perp \gamma[1, c]$ une $K$-forme quadratique, o\`u $\beta=s+t \alpha$ et $\gamma= u+v\alpha$ sont des \'el\'ements de $K^{\times}$ avec $s, t, u, v \in F$. Alors, les deux assertions suivantes sont \'equivalentes:\\(1) $\phi$ est d\'efinie sur $F$.\\(2) $\left<\left<N_{K/F}(\beta), b\right]\right]\cong  \left<\left<N_{K/F}(\gamma), c\right]\right]$, et si $tv\neq 0$, alors $tv \in D_F(\left<\left<N_{K/F}(\beta), b\right]\right])$, o\`u $N_{K/F}$ d\'esigne la norme relative \`a l'extension $K/F$.
\label{pcex2}
\end{prop}

\begin{proof}
On reprend les m\^emes arguments donn\'es dans \cite{LS}. Soit $s:K\longrightarrow F$ l'application $F$-lin\'eaire donn\'ee par: $s(1)=0$ et $s(\alpha)=1$. Soit $s_*$ l'application transfert induite par $s$ au niveau de l'anneau de Witt (ou groupe de Witt). On rappelle le calcul suivant donn\'e dans $W(F)$:
\[s_*(\left<\beta\right>_b)\cong \begin{cases} 0 & \text{si}\; t=0\\t\left<1, N_{K/F}(\beta)\right>_b & \text{si}\; t\neq 0,
\end{cases}\]et de mani\`ere analogue on a
\[s_*(\left<\gamma\right>_b)\cong \begin{cases} 0& \text{si}\; v=0\\v\left<1, N_{K/F}(\gamma)\right>_b & \text{si}\; v\neq 0.
\end{cases}\]
En utilisant la r\'eciprocit\'e de Frobenius, on obtient dans $W_q(F)$:$$s_*(\phi) \cong s_*(\left<\beta\right>_b)\otimes [1, b] \perp s_*(\left<\gamma\right>_b)\otimes [1, c].$$

(1) $\Longrightarrow$ (2) Supposons que $\phi$ soit d\'efinie sur $F$. Alors, $s_*(\phi)=0$ et donc $s_*(\left<\beta\right>_b)\otimes [1, b] \cong s_*(\left<\gamma\right>_b)\otimes [1, c]$.
\begin{itemize}
\item Supposons $tv=0$. On peut supposer $t=0$. Alors, $\beta \in F$ et $s_*(\left<\beta\right>_b)\sim 0 \sim\left<1, N_{K/F}(\beta)\right>_b$ (car $N_{K/F}(\beta)\in F^2$). En particulier, $s_*(\left<\gamma\right>_b)\otimes [1, c]\sim 0 \sim \left<1, N_{K/F}(\gamma)\right>_b\otimes [1, c]$. Ainsi, $\left<\left< N_{K/F}(\beta), b\right]\right]\cong  \left<\left< N_{K/F}(\gamma), c\right]\right]$.

\item Si $tv \neq 0$. Alors, on obtient $t\left<1, N_{K/F}(\beta)\right>_b\otimes [1, b] \cong v\left<1, N_{K/F}(\gamma)\right>_b\otimes [1, c]$. Par la multiplicativit\'e d'une forme de Pfister, on d\'eduit $\left<\left<N_{K/F}(\beta), b\right]\right] \cong \left<\left<N_{K/F}(\gamma), c\right]\right]$ et $tv \in D_F(\left<\left<N_{K/F}(\beta), b\right]\right])$.
\end{itemize}

(2) $\Longrightarrow$ (1) Supposons que les conditions de (2) soient v\'erifi\'ees.

Supposons que $t=0$ (le m\^eme argument s'applique pour le cas $v=0$). Alors, $\beta \in F$ et $\left<\left<N_{K/F}(\beta), b\right]\right] \sim 0 \sim\left<\left<N_{K/F}(\gamma), c\right]\right]=s_*(\gamma [1,c])$. Par cons\'equent, $\gamma [1,c]$ est d\'efinie sur $F$ et donc $\phi$ l'est aussi.

Supposons $tv\neq 0$. Alors, la condition $tv \in D_F(\left<\left<N_{K/F}(\beta), b\right]\right])$ implique que $t\left<\left<N_{K/F}(\beta), b\right]\right]\cong  v\left<\left<N_{K/F}(\gamma), c\right]\right]$. Par cons\'equent, $s_*(\phi)=0$ et donc $\phi$ est d\'efinie sur $F$. \qed
\end{proof}

\medskip

On donne un lemme pr\'eliminaire. 
\begin{lem}
Soient $x$ une ind\'etermin\'ee sur $F$ et $u$ une unit\'e pour la valuation $x$-adique du corps des fractions rationnelles $F(x)$. Alors, pour tout $\epsilon \in \Z$ impair, la forme $[1, ux^{\epsilon}]$ est anisotrope sur 
$F(x)$.
\label{lemiso}
\end{lem}

\begin{proof} Supposons que $\epsilon<0$ et que $[1, ux^{\epsilon}]$ soit isotrope sur $F(x)$. Il existe des polyn\^omes $A, B\in F[x]$ premiers entre eux tels que $A^2+ AB+ ux^{\epsilon}B^2=0$. Ainsi, $x^{-\epsilon}A^2+ x^{-\epsilon}AB= uB^2$. Par cons\'equent, la valuation de $B^2$ est $\geq -\epsilon+1$. En particulier, $x$ divise $B$ et on peut \'ecrire $B=x^{\frac{-\epsilon +1}{2}}C$ pour un certain $C\in F[x]$. Donc, $A^2+ x^{\frac{-\epsilon+1}{2}}AC+ uxC^2=0$. Par cons\'equent, $x$ divise $A$, ce qui est absurde car les polyn\^omes $A$ et $B$ sont suppos\'es premiers entre eux.

Le cas o\`u $\epsilon>0$ se traite de la m\^eme fa\c con en utilisant le fait que $F(x)=F(x^{-1})$ et en travaillant avec la vaulation $x^{-1}$-adique.\qed

\end{proof}
\vskip2mm

Pour le reste de cette section, on prend $F=k(x, y)$ le corps des fractions rationnelles en les variables $x$ et $y$ sur un corps $k$ alg\'ebriquement clos de caract\'eristique $2$. Soit $K=F(\alpha)$, o\`u $\alpha$ est une racine de $Y^2+Y+x\in F[Y]$. On fixe les notations suivantes:\[\tag{$\star$}\begin{cases}
\beta=x\alpha \;\;\text{et}\;\; \gamma= y^{-1}+x\alpha,\\ M=F(\alpha, \wp^{-1}(b), \wp^{-1}(c)),\\a=x,\; b=x^{-2}y^{-1}+1\;\;\text{et}\;\; c=x^{-3}(y^{-2}+xy^{-1}+x^3).
\end{cases}\]
On v\'erifie facilement que l'extension $M/F$ est de degr\'e $8$. Soit $\phi$ la forme quadratique sur $K$ de dimension $4$ donn\'ee par:$$\phi= \beta [1, b] \perp \gamma [1, c].$$

On rappelle une isom\'etrie:
\[\left<\left<r, s\right]\right]\cong \left<\left<r, r+s\right]\right]\; \; \text{pour tous}\; r, s\in F \; \text{avec}\; r\neq 0.
\tag{$\star\star$}
\]

De plus, avec les notations $(\star)$ ci-dessus, on v\'erifie $N_{K/F}(\gamma)=cN_{K/F}(\beta)$ et $b=c+N_{K/F}(\beta)^{-1}y^{-2}$. Ainsi, on obtient
\begin{eqnarray*}
\left<\left<N_{K/F}(\gamma), c\right]\right] & = & \left<\left<cN_{K/F}(\beta), c\right]\right]\\ & \cong & \left<\left<N_{K/F}(\beta), c\right]\right]\:\;\;\; (\text{car}\; c[1, c]\cong [1, c])\\ & \cong & \left<\left<N_{K/F}(\beta)^{-1}y^{-2}, c\right]\right]\\ & \cong & \left<\left<N_{K/F}(\beta)^{-1}y^{-2}, c+N_{K/F}(\beta)^{-1}y^{-2}\right]\right]\:\;\;\; (\text{par }\; (\star\star))\\ & \cong & \left<\left<N_{K/F}(\beta), b\right]\right].
\end{eqnarray*}

Comme $x^2\in \left<\left<N_{K/F}(\beta), b\right]\right]$, alors toutes les conditions de l'assertion (2) de la proposition \ref{pcex2} sont v\'erifi\'ees, et par cons\'equent $\phi$ est d\'efinie sur $F$. Soit $\psi$ une $F$-forme de dimension $4$ telle que $\phi \cong \psi_K$.

\medskip

On a la proposition suivante:
\begin{prop}
La $F$-forme quadratique $\psi$ est hyperbolique sur $M$ mais n'appartient pas  \`a 
$
W(F)\otimes [1, a] + W(F)\otimes [1, b] + W(F)\otimes [1, c].
$
\label{lem2ex}
\end{prop}
\begin{proof}
Soit $L=F(\alpha, \wp^{-1}(x^{-3}y^{-2}))$. Comme $k$ est alg\'ebriquement clos, on a $1\in \wp(k)$, c'est-\`a-dire, $[1, 1]$ est hyperbolique. Il est clair que $[1, c]_L\cong [1, x^{-2}y^{-1}]_L$ puisque $[1, x^{-3}y^{-2}+1]_L\sim 0$. 

L'hyperbolicit\'e de $\psi_M$ est claire car $\psi_M\cong (\psi_K)_M\cong \phi_M\sim 0$. Supposons que $\psi \in W(F)\otimes [1, a] + W(F)\otimes [1, b] + W(F)\otimes [1, c]$. Puisque $\phi \cong \psi_K$ et $[1, c]_L\cong [1, x^{-2}y^{-1}]_L$, on obtient 
\begin{equation}
(\left<\beta, \gamma\right>_b\otimes[1, x^{-2}y^{-1}])_{L}\sim (\rho\otimes [1, x^{-2}y^{-1}])_L
\label{eq1ex}
\end{equation}
pour une certaine $F$-forme bilin\'eaire $\rho$. 

Posons $E=F(\wp^{-1}(x^{-3}y^{-2}))$. Soit $s:L\longrightarrow E$ l'application $E$-lin\'eaire donn\'ee par: $s(1)=0$ et $s(\alpha)=1$. On v\'erifie qu'on a dans $W(E)$ (les calculs sont faits relativement \`a la $E$-base $\{1, \alpha\}$ de $L$):
\vskip1.5mm
(1) $s_*(\left<1\right>_b)\sim 0$ car $s_*(\left<1\right>_b)$ est une forme bilin\'eaire de dimension $2$ isotrope puisque $s(1)=0$. 
\vskip1.5mm
(2) $s_*(\left<\beta\right>_b)=\begin{pmatrix} x & x\\ x & x+x^2\end{pmatrix}\cong \left<1, x\right>_b$.
\vskip1.5mm
(3) $s_*(\left<\gamma\right>_b)=\begin{pmatrix} x & x+y^{-1}\\x+y^{-1} & x+x^2+y^{-1}\end{pmatrix}\cong \left<x, y^{-1}+x^2+x^{-1}y^{-2}\right>_b$.
\vskip1.5mm
Ainsi, en appliquant \`a l'\'equation (\ref{eq1ex}) le transfert $s_*$ induit par $s$, et en utilisant la r\'eciprocit\'e de Frobenius, on obtient:

\begin{equation}
(\left<1, x, x, y^{-1}+x^2+x^{-1}y^{-2}\right>_b\otimes [1, x^{-2}y^{-1}])_{E} \sim s_*(\left<1\right>_b)\otimes\rho[1, x^{-2}y^{-1}]_E\sim 0.
\label{eq2ex}
\end{equation}

Comme $\left<x, x\right>_b$ est un plan m\'etabolique, il est facile de voir que l'\'equation (\ref{eq2ex}) s'\'ecrit:

\begin{equation}
(\left<1, x(1+y^{2}x^3+xy)\right>_b\otimes [1, x^{-2}y^{-1}])_{E} \sim 0.
\label{eq3ex}
\end{equation}

En utilisant l'isom\'etrie $k[1, l]\cong kl[1,l]$, l'\'equation (\ref{eq3ex}) s'\'ecrit

\begin{equation}
(\left<1, xy(1+y^{2}x^3+xy)\right>_b\otimes [1, x^{-2}y^{-1}])_{E} \sim 0.
\label{eq4ex}
\end{equation}

Posons $\zeta=x y(1+y^{2}x^3+xy)$. Notons que $\zeta$ n'est pas un carr\'e dans $F$. Puisque $E/F$ est s\'eparable, alors $\zeta$ n'est pas un carr\'e dans $E$. Par cons\'equent, $\frac{d \zeta}{\zeta}\neq 0$ dans $\nu_E(1)$. En utilisant l'isomorphisme (\ref{Katoiso}), l'isom\'etrie dans (\ref{eq3ex}) se traduit par la relation $\overline{x^{-2}y^{-1}\frac{d \zeta}{\zeta}}=0$ dans $H_2^{2}(E)$. Autrement dit, $\overline{x^{-2}y^{-1}}\in E/\wp(E)$ appartient \`a l'annulateur du symbole $\frac{d \zeta}{\zeta}$ d\'efini par:$${\rm annq}_0(\frac{d \zeta}{\zeta}):=\left\{ \overline{z}\in E/\wp(E)\mid \overline{z\frac{d \zeta}{\zeta}}=0\in H_2^{2}(E)\right\}.$$ Cet annulateur a \'et\'e \'etudi\'e dans \cite{AB3} pour les deux types de symboles (logarithmique et quadratique). En particulier, dans notre cas on a d'apr\`es \cite[Proposition 4.3]{AB3} $\overline{x^{-2}y^{-1}}=\overline{z}\in E/\wp(E)$, o\`u $z=\zeta.A^2$ pour $A\in E^{\times}$ convenable. Ainsi, on a l'isom\'etrie suivante sur le corps $E$:
\begin{equation}
[1, x^{-2}y^{-1}]\cong [1, \zeta.A^2].
\label{eq5ex}
\end{equation}

On affirme que $A\in F$. En effet, on \'ecrit $A=a'+b'\delta$ pour $a', b'\in F$, o\`u $\delta=\wp^{-1}(x^{-3}y^{-2})$. On consid\`ere le transfert ${\rm Tr}_*$ induit par l'application trace de $E/F$. En appliquant ${\rm Tr}_*$ \`a l'isom\'etrie $[1, x^{-2}y^{-1}]\cong [1, \zeta.A^2]$, on d\'eduit que ${\rm Tr}_*([1, \zeta.A^2])=0$ car $[1, x^{-2}y^{-1}]$ est d\'efinie sur $F$. De plus, par \cite[Lemma 34.14]{EKM}, on a ${\rm Tr}_*([1, \zeta.A^2])\equiv [1, {\rm Tr}(\zeta.A^2)]\pmod{I^2_qF}$. On obtient par le Hauptsatz que $[1, {\rm Tr}(\zeta.A^2)]=[1, \zeta. b^{'2}]$ est hyperbolique. Ainsi, on a n\'ecessairement $b'=0$ par le lemme \ref{lemiso} en consid\'erant la valuation $x$-adique de $F$.

Notons que la forme $[1, x^{-2}y^{-1} +\zeta.A^2]$ est anisotrope sur $F$. Sinon, on aurait $[1, x^{-2}y^{-1}]\cong [1, \zeta.A^2]$ et ainsi, par la multiplicativit\'e d'une forme de Pfister, on obtiendrait
\begin{equation}
[1, x^{-2}y^{-1}] \cong \zeta [1, x^{-2}y^{-1}] =  xy(1+y^{2}x^3+xy)[1, x^{-2}y^{-1}] \cong  x(1+y^{2}x^3+xy)[1, x^{-2}y^{-1}].
\label{eqnv1}
\end{equation}

On \'etend l'\'equation (\ref{eqnv1}) au corps $k(x)((y))$ et on prend la valuation $y$-adique. Par le cas (B), la premi\`ere forme r\'esiduelle de $[1, x^{-2}y^{-1}]$ est $\left<1\right>$ alors que celle de $x(1+y^{2}x^3+xy)[1, x^{-2}y^{-1}]$ est $\left< x\right>$, ce qui est absurde.

Puisque $A\in F$ et que la forme $[1, x^{-2}y^{-1}+ \zeta .A^2]$ est anisotrope sur $F$, on d\'eduit de l'\'equation (\ref{eq5ex}) l'isom\'etrie suivante sur $F$:

\begin{equation}
[1, x^{-3}y^{-2} + \zeta.A^2]\cong [1, x^{-2}y^{-1}].
\label{eq6ex}
\end{equation}

On \'etend l'\'equation (\ref{eq6ex}) au corps $k(x)((y))$. Cette forme reste anisotrope sur $k(x)((y))$. On pose $A=y^{\epsilon}u$, o\`u $\epsilon \in \Z$ et $u$ est une unit\'e pour la valuation $y$-adique de $k(x)((y))$. Les deux formes r\'esiduelles de $[1, x^{-2}y^{-1}]$ sont isom\'etriques \`a $\left<1\right>$. Pour le membre de gauche, on montre que la premi\`ere (ou la seconde) forme r\'esiduelle n'est pas isom\'etrique \`a $\left<1\right>$, ce qui donne une contradiction. Pour cela on discute sur l'entier $\epsilon$:
\vskip1.5mm

\noindent{\bf Cas 1.} Supposons $\epsilon \geq -1$. Alors, $2\epsilon +3>0$. On a
\begin{eqnarray*}
[1, x^{-3}y^{-2} + \zeta.A^2] &=& [1, y^{-2}(x^{-3}+x(1+y^2x^3+ xy)y^{2\epsilon+3}u^2)].
\end{eqnarray*}
Par le cas (C), la premi\`ere forme r\'esiduelle de $[1, x^{-3}y^{-2} + \zeta.A^2]$ est $\left<1, x\right>$.

\vskip1.5mm

\noindent{\bf Cas 2.} Supposons $\epsilon <-1$. Alors, $-2-(2\epsilon+1)>0$. On a
\begin{eqnarray*}
[1, x^{-3}y^{-2} + \zeta.A^2] &=& [1, y^{2\epsilon +1}(x^{-3}y^{-2-(2\epsilon+1)}+x(1+y^2x^3+ xy)u^2)].
\end{eqnarray*}
Par le cas (B), la seconde forme r\'esiduelle de $[1, x^{-3}y^{-2} + \zeta.A^2]$ est $\left<x \overline{u}^2\right>\cong \left<x\right>$.

Ceci ach\`eve la preuve de la proposition.\qed
\end{proof}

\medskip
On obtient le corollaire suivant:

\begin{cor}
L'extension biquadratique $F(\wp^{-1}(a), \wp^{-1}(b))/F$ n'est pas excellente.
\label{corexce}
\end{cor}

\begin{proof} On a $\Delta(\psi_K)=b+c\in K/\wp(K)$. Ainsi, $\Delta(\psi)=\epsilon a+b+c\in F/\wp(F)$ pour $\epsilon=0$ ou $1$. Puisque $[1, a]_K\sim 0$, on peut supposer que $\Delta(\psi)=a+b+c\in F/\wp(F)$. La forme $\psi$ est anisotrope sur $F$ car sinon il existerait $r\in F^{\times}$ tel que $\psi\sim r[1, a+b+c]\sim r[1, a]\perp r[1, b]\perp r[1, c]$, ce qui est exclu par la proposition \ref{lem2ex}. 

Posons $L=F(\wp^{-1}(a), \wp^{-1}(b))$. Comme $M=L(\wp^{-1}(c))$, $\psi_M$ est hyperbolique et $\Delta(\psi_L)=c\in L/\wp(L)$, on d\'eduit que $\psi_L$ est isotrope \cite[Lemme 4.3]{L02}. Suppsons que $(\psi_L)_{an}$ soit d\'efinie sur $F$. Alors, il existe $r\in F^{\times}$ tel que $\psi_L \sim r[1, c]_L$. Ainsi, $(\psi \perp r[1, c])_L\sim 0$. Par \cite[Corollary 4.16]{B1}, il existe $\rho, \rho'$ des $F$-formes bilin\'eaires telles que $\psi \perp r[1, c] \sim \rho\otimes[1, a] \perp \rho'\otimes [1, b]$, ce qui n'est pas possible par la proposition \ref{lem2ex}. Ceci montre que l'extension $L/F$ n'est pas excellente.\qed
\end{proof}
\medskip

Les premiers exemples des extensions biquadratiques qui ne sont pas excellentes ont \'et\'e donn\'es dans \cite[Example 5.8]{ELTW}. En caract\'eristique $2$ la situation est diff\'erente. On sait que les extensions biquadratiques purement ins\'eparables sont excellentes par un r\'esultat de Hoffmann \cite{H}. R\'ecemment, il a \'et\'e prouv\'e dans \cite{LM} que toute extension biquadratique mixte est excellente \cite{LM}. En g\'en\'eral une extension biquadratique s\'eparable n'est pas excellente. Cela a \'et\'e fait dans \cite{LM} en g\'en\'eralisant \`a la caract\'eristique $2$ un exemple de Sivatski. Cette g\'en\'eralisation repose sur un r\'esultat de Rowen donnant l'existence d'une alg\`ebre centrale \`a division ind\'ecomposable de degr\'e $8$ et d'exposant $2$. Le corollaire \ref{corexce} donne un nouveau exemple de la non-excellence des extensions biquadratiques s\'eparables en caract\'eristique $2$.

\medskip

\noindent{\bf Preuve de la proposition \ref{pcex}.} Soit $\psi$ comme dans la proposition \ref{lem2ex}. On consid\`ere la forme $\psi'=\psi \perp [1, \Delta(\psi)]\perp [0,0]$ de dimension $8$ appartenant \`a $I^2_qF$. On a $C(\psi')\cong M_2(A)$ pour $A$ une $F$-alg\`ebre de degr\'e $8$. Puisque $\psi$ est anisotrope (voir la preuve du corollaire \ref{corexce}), l'alg\`ebre $A$ est d'indice $4$ (corollaire \ref{cordec}). Puisque $\psi_M$ est hyperbolique, alors $A_M$ est d\'eploy\'ee. Mais $A$ n'admet pas de d\'ecomposition adapt\'ee \`a $M$ car sinon il existerait une forme $\psi''\in I^2_qF$ v\'erifiant: $\psi''\in W(F)\otimes [1, a] + W(F)\otimes [1, b] + W(F)\otimes [1, c]$ et $C(\psi'')\sim A\sim C(\psi')$. Ainsi, on aurait $\psi' \perp \psi'' \in I^3_qF=0$. Par cons\'equent, $\psi \in W(F)\otimes [1, a] + W(F)\otimes [1, b] + W(F)\otimes [1, c]$, ce qui est exclu par la proposition \ref{lem2ex}.\qed

\medskip

L'alg\`ebre de degr\'e $8$ et d'exposant $2$ construite par Elman-Lam-Tignol-Wadsworth en caract\'eristique diff\'erente $2$ et son analogue en caract\'eristique $2$ que nous donnons dans la proposition \ref{pcex} ont toutes les deux un centre de la forme $k(x,y)$ o\`u $k$ est un corps alg\'ebriquement clos. Il est bien connu que sur de tels corps $k(x,y)$, les alg\`ebres simples centrales d'exposant $2$ sont toutes d'indice au plus $2$. Il devient assez naturel de se poser la question suivante: soient $D$ une alg\`ebre \`a division de degr\'e $8$ et d'exposant $2$ sur un corps $F$ v\'erifiant $\operatorname{cd}_2(F)=2$, et $K/F$ une extension triquadratique s\'eparable telle que $D_K$ soit d\'eploy\'ee.
\begin{ques}
 L'alg\`ebre $D$ admet-elle une d\'ecomposition adapt\'ee \`a $K$?
\end{ques}

\end{sloppypar}
\end{document}